\documentclass[a4paper, 12point]{article}
\usepackage{}
\newcommand{\be}{\begin{equation}}
\newcommand{\ee}{\end{equation}}
\newcommand{\bea}{\begin{eqnarray}}
\newcommand{\eea}{\end{eqnarray}}

\usepackage{graphicx,epsfig}
\usepackage{enumerate}
\usepackage{amssymb}
\usepackage{mathrsfs}
\usepackage{amsthm}
\usepackage{amsmath}
\usepackage{setspace}
\usepackage{endnotes}
\newtheorem{thm}{Theorem}[section]
\theoremstyle{definition}

\theoremstyle{lemma}

\theoremstyle{example}

\newtheorem{rmrk}{Remark}[section]
\theoremstyle{illustration}

\theoremstyle{proposition}

\theoremstyle{corollary}
\newtheorem{cor}{Corollary}[section]
\numberwithin{equation}{section}
\begin{document}
\date{}
\title{\textbf{Difference Sequence Spaces Derived by Using Generalized Means}}
\author{Atanu Manna\footnote{Corresponding author e-mail: atanumanna@maths.iitkgp.ernet.in}, Amit Maji\footnote{Author's e-mail: amit.iitm07@gmail.com}, P. D. Srivastava \footnote{Author's e-mail: pds@maths.iitkgp.ernet.in}\\
\textit{\small{Department of Mathematics, Indian Institute of Technology, Kharagpur}} \\
\textit{\small{Kharagpur 721 302, West Bengal, India}}}
\maketitle
\vspace{20pt}
\begin{center}\textbf{Abstract}\end{center}
This paper deals with new sequence spaces $X(r, s, t ;\Delta) $ for $X\in \{l_\infty, c, c_0\}$ defined by using generalized means and difference operator. It is shown that these spaces are complete normed linear spaces and the spaces $X(r, s, t ;\Delta)$ for $X\in \{c, c_0\}$ have Schauder basis. Furthermore, the $\alpha$-, $\beta$-, $\gamma$- duals of these sequence spaces are computed and also established necessary and sufficient conditions for matrix transformations from $X(r, s, t ;\Delta)$ to $X$.

\section{Introduction}
The study of sequence spaces has importance in the several branches of analysis, namely, the structural theory of topological vector spaces, summability theory, Schauder basis theory etc. Besides this, the theory of sequence spaces is a powerful tool for obtaining some topological and geometrical results using Schauder basis.

Let $w$ be the space of all real or complex sequences $x=(x_n)$, $n\in \mathbb{N}_0$. For an infinite matrix $A$ and a sequence space $\lambda$, the matrix domain of $A$, which is denoted by $\lambda_{A}$ and defined as $\lambda_A=\{x\in w: Ax\in \lambda\}$ \cite{WIL}. Basic methods, which are used to determine the topologies, matrix transformations and inclusion relations on sequence spaces can also be applied to study the matrix domain $\lambda_A$. In recent times, there is an approach of forming new sequence spaces by using matrix domain of a suitable matrix and characterize the matrix mappings between these sequence spaces.

Kizmaz first introduced and studied the difference sequence spaces in \cite{KIZ}. Later on,
several authors including Ahmad and Mursaleen \cite{AHM}, \c{C}olak
and Et \cite{COL}, Ba\c{s}ar and Altay \cite{ALT}, Orhan \cite{ORH},
Polat and Altay \cite{POL1}, Aydin and Ba\c{s}ar \cite{AYD} etc. have introduced and studied new sequence spaces
defined by using difference operator.

On the other hand, sequence spaces are also defined by using
generalized weighted means. Some of them can be viewed in Malkowsky
and Sava\c{s} \cite{MAL}, Altay and Ba\c{s}ar \cite{ALT1}. Mursaleen and Noman \cite{MUR} introduced a sequence
space of generalized means, which includes most of the earlier known
sequence spaces. But till $2011$, there was
no such literature available in which a sequence space is generated by
combining both the weighted means and the difference operator. This
was firstly initiated by Polat et al. \cite{POL}. The authors in \cite{POL} have
introduced the sequence spaces $\lambda(u, v; \Delta)$
for $\lambda \in \{ l_{\infty}, c, c_{0} \}$ defined as
\begin{center}
$\lambda(u, v; \Delta)= \big\{x\in w : (G(u, v).\Delta) x\in
\lambda\big\}$,
\end{center}where $u, v\in w$ such that $u_n, v_n \neq 0$ for all $n$ and the matrices $G(u, v)=(g_{nk})$, $\Delta=(\delta_{nk})$ are defined by
\begin{align*}
g_{nk} &= \left\{
\begin{array}{ll}
    u_nv_k & \quad \mbox{~if~} 0\leq k \leq n,\\
    0 & \quad \mbox{~if~} k > n
\end{array}\right.&
\delta_{nk}& = \left\{
\begin{array}{ll}
    0 & \quad \mbox{~if~} 0\leq k <n-1 \\
    (-1)^{n-k} & \quad \mbox{~if~} n-1\leq k \leq n,\\
    0 & \quad \mbox{~if~} k>n,
\end{array}\right.
\end{align*}
respectively.

The aim of this article is to introduce new sequence spaces
defined by using both the generalized means and the difference
operator. We investigate some topological properties as well
as $\alpha$-, $\beta$-, $\gamma$- duals and bases of the new
sequence spaces are obtained. Further, we characterize some matrix
transformations between these new sequence spaces.
\section{Preliminaries}
Let $l_\infty, c$ and $c_0$ be the spaces of all bounded, convergent and null sequences $x=(x_n)$ respectively, with norm $\|x\|_\infty=\displaystyle\sup_{n}|x_n|$. Let  $bs$ and $cs$ be the sequence spaces of all bounded and convergent series respectively. We denote by $e=(1, 1, \cdots)$ and $e_{n}$ for the sequence whose $n$-th term is $1$ and others are zero and $\mathbb{{N_{\rm 0}}}=\mathbb{N}\cup \{0\}$, where $\mathbb{N}$ is the set of all natural numbers.
A sequence $(b_n)$ in a normed linear space $(X,
\|.\|)$ is called a Schauder basis for $X$ if for every $x\in X$
there is a unique sequence of scalars $(\mu_n)$ such that
\begin{center}
$\Big\|x-\displaystyle\sum_{n=0}^{k}\mu_nb_n\Big\|\rightarrow0$ as $k\rightarrow\infty$,
\end{center}
i.e., $x=\displaystyle\sum_{n=0}^{\infty}\mu_nb_n$.
 (\cite{WIL}, \cite{MAD}).\\
For any subsets $U$ and $V$ of $w$, the multiplier space $M(U, V)$ of $U$ and $V$ is defined as
\begin{center}
$M(U, V)=\{a=(a_n)\in w : au=(a_nu_n)\in V ~\mbox{for all}~ u\in U\}$.
\end{center}
In particular,
\begin{center}
$U^\alpha= M(U, l_1)$, $U^\beta= M(U, cs)$ and $U^\gamma= M(U, bs)$
\end{center} are called the $\alpha$-, $\beta$- and $\gamma$- duals of $U$ respectively \cite{MAL1}.

Let $A=(a_{nk})_{n, k}$ be an infinite matrix
with real or complex entries $a_{nk}$. We write $A_n$ as the
sequence of the $n$-th row of $A$, i.e.,
$A_n=(a_{nk})_{k}$ for every $n$.
For $x=(x_n)\in w$, the $A$-transform of $x$ is defined as the
sequence $Ax=((Ax)_n)$, where
\begin{center}
$A_n(x)=(Ax)_n=\displaystyle\sum_{k=0}^{\infty}a_{nk}x_k$,
\end{center}
provided the series on the right side converges for each $n$. For any two sequence spaces $U$ and $V$, we denote by $(U, V)$, the class of all infinite matrices $A$ that map $U$ into $V$. Therefore $A\in (U, V)$ if and only if $Ax=((Ax)_n)\in V$ for all $x\in U$. In other words, $A\in (U, V)$ if and only if $A_n \in U^\beta$ for all $n$ \cite{WIL}.
An infinite matrix $T={(t_{nk})}_{n,k}$ is said to be triangle if $t_{nk}=0$ for $k>n$ and $t_{nn}\neq 0$, $n\in \mathbb{{N_{\rm 0}}}$.
\section{Sequence spaces $X(r, s, t; \Delta)$ for $X \in \{ l_{\infty}, c, c_{0}\}$}
In this section, we first begin with the notion of generalized means given by Mursaleen et al. \cite{MUR}.\\

We denote the sets $\mathcal{U}$ and $\mathcal{U}_{0}$ as
\begin{center}
$ \mathcal{U} = \Big \{ u =(u_{n})_{n=0}^{\infty} \in w: u_{n} \neq
0~~ {\rm for~ all}~~ n \Big \}$ and $ \mathcal{U_{\rm 0}} = \Big \{ u
=(u_{n})_{n=0}^{\infty} \in w: u_{0} \neq 0 \Big \}.$
\end{center}
Let $r, t \in \mathcal{U}$ and $s \in \mathcal{U}_{0}$. The sequence $y=(y_{n})$ of generalized means of a sequence $x=(x_{n})$ is defined
by $$ y_{n}= \frac{1}{r_{n}}\sum_{k=0}^{n} s_{n-k}t_{k}x_{k} \qquad (n \in \mathbb{N_{\rm 0}}).$$
The infinite matrix $A(r, s, t)$ of generalized means is defined by

$$(A(r,s,t))_{nk} = \left\{
\begin{array}{ll}
    \frac{s_{n-k}t_{k}}{r_{n}} & \quad 0\leq k \leq n,\\
    0 & \quad k > n.
\end{array}\right. $$

Since $A(r, s, t)$ is a triangle, it has a unique inverse and the
inverse is also a triangle \cite{JAR}. Take $D_{0}^{(s)} =
\frac{1}{s_{0}}$ and

$ D_{n}^{(s)} =
\frac{1}{s_{0}^{n+1}} \left|
\begin{matrix}
    s_{1} & s_{0} &  0 & 0 \cdots & 0 \\
    s_{2} & s_{1} & s_{0}& 0 \cdots & 0 \\
    \vdots & \vdots & \vdots & \vdots    \\
    s_{n-1} & s_{n-2} & s_{n-3}& s_{n-4} \cdots & s_0 \\
      s_{n} & s_{n-1} & s_{n-2}& s_{n-3} \cdots & s_1
\end{matrix} \right| \qquad \mbox{for}~ n =1, 2, 3, \cdots $ \\ \\

Then the inverse of $A(r, s, t)$ is the triangle $B= (b_{nk})_{n, k}$ which is defined as
$$b_{nk} = \left\{
\begin{array}{ll}
    (-1)^{n-k}~\frac{D_{n-k}^{(s)}}{t_{n}}r_{k} & \quad 0\leq k \leq n,\\
    0 & \quad k > n.
\end{array}\right. $$
We now introduce the sequence spaces $X(r, s, t; \Delta)$ for $X \in \{ l_{\infty}, c, c_{0} \}$ as
$$ X(r, s,t; \Delta)= \Big \{ x=(x_{n})\in w : \Big( \frac{1}{r_{n}}\sum_{k=0}^{n}s_{n-k}t_{k} \Delta {x_{k}}\Big)_{n} \in X \Big  \},$$
which is a combination of the generalized means and the difference operator $\Delta$ such that $\Delta x_k=x_k-x_{k-1}$, $x_{-1}=0$. By using matrix domain, we can write $X(r, s,t; \Delta)=  X_{A(r, s,t; \Delta)}=\{x \in w :A(r, s,t; \Delta)x\in X\}$, where $A(r, s,t; \Delta)= A(r, s,t). \Delta$, product of two triangles $A(r, s,t)$ and $\Delta$.\\ \\
These sequence spaces include many known sequence spaces studied by several authors. For examples,
\begin{enumerate}[I.]
\item if $r_{n}=\frac{1}{u_{n}}$, $t_{n}=v_{n}$ and $s_{n}=1~ \forall~ n$, then the sequence spaces $X(r, s,t; \Delta)$ reduce to $ X(u, v; \Delta)$ for $X \in \{ l_{\infty}, c, c_{0} \}$ introduced and studied by Polat et al. \cite{POL}

\item if $t_{n}=1$, $s_{n}=1$ ~$\forall~ n$ and $r_{n}=n+1$, then the sequence space $ X(r, s,t; \Delta)$ for $X=l_\infty$ reduces to $ ces_\infty(\Delta)$ studied by Orhan \cite{ORH}.

\item if $r_{n}= \frac{1}{n!}, $ $t_{n}=\frac{\alpha^n}{n!}$, $s_{n}=\frac{(1-\alpha)^n}{n!}$, where $0<\alpha<1$, then the sequence spaces $ X(r, s,t; \Delta)$ for $X \in \{l_\infty, c, c_0 \}$ reduce to $e^\alpha_\infty(\Delta)$,
$e^\alpha_c(\Delta)$ and $e^\alpha_0(\Delta)$ respectively \cite{POL1}.
\item if $r_{n}=n+1,$ $t_{n}={1+ \alpha^n}$, where $0<\alpha<1$ and $s_{n}=1~\forall n$, then the sequence spaces $ X(r, s,t; \Delta)$ for $X \in \{c, c_0 \}$
 reduce to the spaces of sequences $a_{c}^{\alpha}(\Delta)$ and $a_{0}^{\alpha}(\Delta)$ studied by Aydin and Ba\c{s}ar \cite{AYD}. For $X=l_\infty$, the sequence space $ X(r, s,t; \Delta)$ reduces to $a_{\infty}^{\alpha}(\Delta)$ studied by Djolovi\'{c} \cite{DJO}.
\end{enumerate}
\section{Main results}
In this section, we begin with some topological results of the
newly defined sequence spaces.
\begin{thm}
The sequence spaces $X(r,s, t; \Delta)$ for $X\in \{l_\infty, c, c_0
\}$ are complete normed linear spaces under the norm defined by
\begin{center}
$\|x\|_{X(r, s, t; \Delta)}=\displaystyle\sup_n\Big|\frac{1}{r_n}\displaystyle\sum_{k=0}^{n}s_{n-k}t_k\Delta x_k\Big|=\displaystyle\sup_n|(A(r, s, t; \Delta)x)_n| $.
\end{center}
\end{thm}
\begin{proof}
Let $u$, $v\in X(r, s, t; \Delta)$ and $\alpha, \beta$ be any two
scalars. Then
\begin{center}
$\displaystyle\sup_{n}\bigg|\frac{1}{r_n}\displaystyle\sum_{k=0}^{n}s_{n-k}t_k\Delta(\alpha
u_k + \beta v_k) \bigg|\leq
|\alpha|\displaystyle\sup_{n}\bigg|\frac{1}{r_n}\displaystyle\sum_{k=0}^{n}s_{n-k}t_k\Delta
u_k\bigg| +
|\beta|\displaystyle\sup_{n}\bigg|\frac{1}{r_n}\displaystyle\sum_{k=0}^{n}s_{n-k}t_k\Delta
v_k\bigg|$
\end{center}
and hence $\alpha u + \beta v\in X(r, s, t; \Delta)$. Therefore $X(r, s, t; \Delta)$ is a linear space. It is easy to show that the functional $\|.\|_{X(r, s, t; \Delta)}$ defined above gives a norm on the linear space $X(r, s, t; \Delta)$.\\
To show completeness, let $(x^m)$ be a Cauchy sequence in $X(r, s, t; \Delta)$, where $x^{m}= (x_k^{m})=(x_0^{m}, x_1^{m}, x_2^{m}, \ldots )$ $\in X(r, s, t; \Delta)$, for each $m\in \mathbb{N}_0$. Then for every $\epsilon>0$ there exists $n_0\in \mathbb{N}$ such that
\begin{center}
$\|x^{m}-x^{l}\|_{X(r, s, t; \Delta)}<\epsilon$ ~ for $m, l\geq n_0$.
\end{center}
The above implies that for each $k\in \mathbb{N_{\rm 0}}$,
\begin{equation}
|A(r, s, t; \Delta)(x_k^{m}-x_k^{l})|<\epsilon  ~~\mbox{for all}~ m, l\geq n_0.
\end{equation}
Therefore $((A(r, s, t). \Delta)x_k^{m})$ is a Cauchy
sequence of scalars for each $k\in \mathbb{N_{\rm 0}}$ and hence $((A(r, s, t). \Delta)x_k^{m})$ converges for each $k$. We write
\begin{center}
$\displaystyle\lim_{m\rightarrow\infty}((A(r, s, t). \Delta)x_k^{(m)}) =
((A(r, s, t). \Delta)x_k)$, \quad $k\in \mathbb{N}_0$.
\end{center}
Letting $l\rightarrow\infty$ in $(4.1)$, we obtain
\begin{equation}
\Big|A(r, s, t; \Delta)(x_k^{m}-x_k)\Big|<\epsilon  ~\mbox{for
all}~ m \geq n_0 ~\mbox{and each}~ k\in \mathbb{N_{\rm
0}}.
\end{equation}
Hence by definition $\|x^{m}-x\|_{X(r, s, t; \Delta)}<\epsilon$
for all $m\geq n_0$. Next we show that $x \in X(r,s,t;
\Delta)$. Consider
\begin{center}
$\|x\|_{X(r, s, t; \Delta)}\leq \|x^{m}\|_{X(r, s, t; \Delta)} + \|x^{m}-x\|_{X(r, s, t; \Delta)}$,
\end{center}
which is finite for $m\geq n_0$ and hence $x\in X(r, s, t; \Delta)$. This completes the proof.
\end{proof}

\begin{thm}
The sequence spaces $X(r,s,t; \Delta)$, where $X \in \{ l_{\infty}, c, c_{0} \}$ are linearly isomorphic to the spaces $X \in \{ l_{\infty}, c, c_{0} \}$
respectively i.e. $l_{\infty}(r, s, t; \Delta) \cong l_{\infty}$,  $c(r, s, t; \Delta) \cong c$ and $c_{0}(r, s, t; \Delta) \cong c_{0}$.
\end{thm}
\begin{proof}
We prove the theorem only for the case $X=l_{\infty}$. To prove this, we need to show that there exists a bijective linear map from $l_{\infty}(r,s,t;\Delta)$ to $l_{\infty}$.\\
 We define a map $T:l_{\infty}(r,s,t;\Delta)\rightarrow l_{\infty}$ by $x\longmapsto Tx=y=(y_n)$, where
$$y_n= \frac{1}{r_{n}}\sum_{k=0}^{n}s_{n-k}t_{k} \Delta {x_{k}}.$$
Since $\Delta$ is a linear operator, so the linearity of $T$ is trivial. It is clear from the definition
that $Tx=0$ implies $x=0$. Thus $T$ is injective. To prove $T$ is surjective, let $y =(y_n)\in l_{\infty}$.
Since $y= (A(r, s, t). \Delta)x$, i.e.,
$$x= \big(A(r, s, t). \Delta\big)^{-1}y=\Delta^{-1}.A(r, s, t)^{-1}y, $$
we can get a sequence $x=(x_{n})$ as
\begin{equation}
x_n = \sum_{j=0}^{n}\sum_{k=0}^{n-j} (-1)^{k} \frac{D_{k}^{(s)}}{t_{k+j}}r_{j}y_{j},  \qquad n \in \mathbb{N_{\rm 0}}.
\end{equation}
Then
$$\|x \|_{l_{\infty}(r,s,t; \Delta)}= \sup_{n }\bigg| \frac{1}{r_{n}}\sum_{k=0}^{n}s_{n-k}t_{k} \Delta {x_{k}}\bigg| = \sup_{n}|y_{n}|= \|y \|_{\infty} < \infty.$$
Thus $x \in l_{\infty}(r, s, t; \Delta)$ and this shows that $T$ is
surjective. Hence $T$ is a linear bijection from $l_{\infty}(r,s,t;
\Delta)$ to $l_{\infty}$. Also $T$ is norm preserving. So $l_{\infty}(r, s, t; \Delta) \cong l_{\infty}$.\\
In the same way, we can prove that $c_{0}(r, s, t; \Delta) \cong c_{0}$, $c(r, s, t; \Delta) \cong c$. This completes the proof.
\end{proof}
Since $X(r,s,t; \Delta) \cong X$ for $X \in \{c_{0}, c\}$, the Schauder bases of the sequence spaces $X(r,s,t; \Delta)$ are the inverse image of the bases of $X$ for $X \in \{c_{0}, c\}$. So, we have the following theorem without proof.
\begin{thm}
Let $\mu_k=(A(r,s, t; \Delta)x)_k$ for all $k\in \mathbb{N_{\rm
0}}$. Define the sequences $b^{(j)}=(b_{n}^{(j)}), j\in \mathbb{N}_0$ and $b_n^{(-1)}$ as
\begin{align*}
b_n^{(j)} &= \left\{
\begin{array}{ll}
 \displaystyle\sum_{k=0}^{n-j} (-1)^k \frac{D_k^{(s)}}{t_{k+j}}r_j& \mbox{if}~~   0\leq j \leq n \\
0 & \mbox{if} ~~ j>n.
\end{array}\right.& \mbox{and~~}
b_n^{(-1)} & = \displaystyle\sum_{j=0}^{n} \displaystyle\sum_{k=0}^{n-j} (-1)^k \frac{D_k^{(s)}}{t_{k+j}}r_j.
\end{align*}

Then the followings are true:\\
 $(i)$ The sequence $(b^{(j)})_{j=0}^{\infty}$ is a basis for the space $c_0(r, s,t; \Delta)$ and any sequence $x\in c_0(r, s,t; \Delta)$ has a unique representation of the form
\begin{center}
$x=\displaystyle\sum_{j=0}^{\infty}\mu_jb^{(j)}$.
\end{center}
$(ii)$ The sequence $(b^{(j)})_{j=-1}^{\infty}$ is a basis for the space $c(r,
s,t; \Delta)$ and any $x\in c(r, s,t; \Delta)$ has a unique
representation of the form
\begin{center}
$x=\ell b_n^{(-1)} + \displaystyle\sum_{j=0}^{\infty}(\mu_j-\ell)b^{(j)}$,
\end{center}where $\ell=\displaystyle\lim_{n\rightarrow\infty}(A(r,s, t; \Delta)x)_n$.
\end{thm}
\begin{rmrk}
In particular, if we choose $r_{n}=\frac{1}{u_{n}}$,
$t_{n}=v_{n}$, $s_{n}=1$, $\forall~n$ then the sequence spaces $ X(r,
s,t; \Delta)$ reduce to $ X(u, v; \Delta)$ for $X \in \{
l_{\infty}, c, c_{0} \}$ \cite{POL}. With this choice of $(s_{n})$,
we have $D_{0}^{(s)} =D_{1}^{(s)}=1$ and $D_{n}^{(s)} =0 $ for $ n
\geq 2$. Thus the sequences $b^{(j)}=(b_{n}^{(j)}), j\in \mathbb{N}_0$ and $b_n^{(-1)}$ reduce to
\begin{align*}
 b_{n}^{(j)}&= \left\{
     \begin{array}{ll}
        {  \frac{1}{u_j}\Big(\frac{1}{v_j} - \frac{1}{v_{j+1}}\Big) }& \mbox{if}~~ \quad 0\leq j < n \\
        \frac{1}{u_{n}v_{n}}              & \mbox{if}~~  \quad  j=n\\
        0                                 & \mbox{if}~~ \quad j > n.
        \end{array}\right.& \mbox{and~~}
b_n^{(-1)} & = \displaystyle\sum_{j=0}^{n-1} \frac{1}{u_j}\Big(\frac{1}{v_j}-\frac{1}{v_{j+1}}\Big)+ \frac{1}{u_{n}v_{n}}  .
\end{align*}
The sequences $(b^{(j)})_{j=0}^{\infty}$ and $(b^{(j)})_{j=-1}^{\infty}$ are the bases for the spaces $c_{0}(u,v; \Delta)$ and $c(u,v; \Delta)$ respectively \cite{POL}.
\end{rmrk}

\par Let $\mathcal{F}$ be the collection of all nonempty finite subsets of the set of all natural numbers and
$A=(a_{nk})_{n,k}$ be an infinite matrix satisfying the conditions:\\
\begin{align}
&\sup_{K \in \mathcal{F}} \sum_{n=0}^{\infty}\Big|\sum_{k \in K}a_{nk}\Big| < \infty\\
&\displaystyle\sup_{n} \sum_{k=0}^{\infty}|a_{nk}| < \infty
\end{align}
\begin{align}
&\displaystyle\lim_{n} \sum_{k=0}^{\infty}|a_{nk}| =0\\
&\displaystyle\lim_{n} a_{nk} =0 \mbox{~for all~} k\\
&\displaystyle\lim_{n} \sum_{k=0}^{\infty}a_{nk}=0\\
&\displaystyle\lim_{n} a_{nk} \mbox{~exists for all~} k\\
&\displaystyle\lim_{n} \sum_{k=0}^{\infty}|a_{nk}- \lim_n a_{nk}| =0\\
&\displaystyle\lim_{n} \sum_{k=0}^{\infty}a_{nk} \mbox{~exists~}
\end{align}

We now state some results of Stieglitz and Tietz \cite{STI} which are required to obtain the duals and to characterize some matrix transformations.
\begin{thm} \cite{STI}
$(a)$ $A \in (c_0, l_1), A \in (c,l_1), A \in (l_{\infty}, l_1)$ if and only if $(4.4)$ holds.\\
$(b)$ $A \in (c_{0}, l_{\infty}), A \in (c, l_{\infty}), A \in (l_{\infty}, l_{\infty})$ if and only if $(4.5)$ holds.\\
$(c)$ $A \in (c_0, c_0)$ if and only if $(4.5)$ and $(4.7)$ hold.\\
$(d)$ $A \in (l_{\infty}, c_0)$ if and only if $(4.6)$ holds.\\
$(e)$ $A \in (c,c_0)$ if and only if $(4.5)$, $(4.7)$ and $(4.8)$ hold.\\
$(f)$ $A \in (c_0, c)$ if and only if $(4.5)$ and $(4.9)$ hold.\\
$(g)$ $A \in (l_{\infty}, c)$ if and only if $(4.5)$, $(4.9)$ and $(4.10)$ hold.\\
$(h)$ $A \in (c, c)$ if and only if $(4.5)$, $(4.9)$ and $(4.11)$ hold.
\end{thm}

\subsection{The $\alpha$-dual, $\gamma$-dual of $X(r,s, t; \Delta)$ for $X\in\{l_\infty, c, c_0 \}$}
\begin{thm}
The $\alpha$-dual of the space $X(r,s,t; \Delta)$ for $X\in\{l_\infty, c, c_0\}$ is the set
$$ \Lambda = \Big \{ a=(a_{n}) \in w: \sup_{K \in \mathcal{F}} \sum_{n}\Big|\sum_{j \in K}\sum_{k=0}^{n-j} (-1)^{k} \frac{D_{k}^{(s)}}{t_{k+j}}r_{j}a_{n}\Big| < \infty \Big \}.$$
\end{thm}
\begin{proof}
Let $a=(a_{n}) \in w$, $x\in X (r, s, t; \Delta)$ and $y\in X$ for $X \in \{ l_{\infty}, c, c_{0}\}$. Then for each $n \in \mathbb{N_{\rm 0}}$, we have
$$  a_{n}x_n = \sum_{j=0}^{n}\sum_{k=0}^{n-j} (-1)^{k} \frac{D_{k}^{(s)}}{t_{k+j}}r_{j}a_{n}y_{j} =(Cy)_{n},$$
where the matrix $C=(c_{nj})$ is defined as
$$
c_{nj} = \left\{
\begin{array}{ll}
 \displaystyle\sum_{k=0}^{n-j} (-1)^k \frac{D_k^{(s)}}{t_{k+j}}r_ja_n& \mbox{if}~~   0\leq j \leq n \\
0 & \mbox{if} ~~ j>n
\end{array}
\right.
$$
and $x_n$ is given by $(4.3)$.
Thus for each $x \in X(r,s,t; \Delta)$, $(a_nx_{n})_{n} \in l_{1}$ if
and only if $(Cy)_{n} \in l_{1}$ where $y \in X $ for $X  \in
\{l_{\infty}, c, c_{0} \}$. Therefore $a=(a_{n}) \in [X(r,s,t;
\Delta)]^{\alpha}$ if and only if $C \in (X, l_1)$. By using Theorem 4.4(a)
, we have
$$ [X(r,s,t; \Delta)]^{\alpha} = \Lambda.$$
\end{proof}

\begin{thm}
The $\gamma$-dual of $X(r, s, t;\Delta)$ for $X\in\{l_\infty, c, c_0\}$ is the set $$\Gamma = \Big\{a=(a_n)\in w: ~~ \displaystyle\sup_{m}\displaystyle\sum_{n=0}^{\infty}|e_{mn}|<\infty\Big\},$$
where the matrix $E=(e_{mn})$ is defined by
$$E=(e_{mn})= \left\{
\begin{array}{ll}
   r_n \displaystyle\bigg[\frac{a_{n}}{s_0t_n} + \Big(\frac{D_{0}^{(s)}}{t_n}- \frac{D_{1}^{(s)}}{t_{n+1}} \Big)\sum_{j=n+1}^{m}a_{j}+
   \sum_{j=n+2}^{m}(-1)^{j-n} \frac{D_{j-n}^{(s)}}{t_{j}}\Big(\sum_{k=j}^{m}a_{k}\Big) \bigg] & \quad 0\leq n \leq m,\\
    0 & \quad n > m.
\end{array}\right. $$
Note: We mean $\sum\limits_{j =n}^{m} =0$ if $n > m$.
\end{thm}

\begin{proof}
Let $a =(a_n)\in w$, $x \in X(r,s,t; \Delta)$ and $y \in X$, where $X \in \{l_{\infty}, c,
c_{0} \}$. Then by using (4.3), we have
\begin{align*}
\displaystyle\sum_{n=0}^{m}a_n x_n & =\sum_{n=0}^{m}\sum_{j=0}^{n}\sum_{k=0}^{n-j}(-1)^{k} \frac{D_{k}^{(s)}r_{j}y_{j}a_{n}}{t_{k+j}}\\
& = \sum_{n=0}^{m-1}\sum_{j=0}^{n}\sum_{k=0}^{n-j}(-1)^{k} \frac{D_{k}^{(s)}r_{j}y_{j}a_{n}}{t_{k+j}}+ \sum_{j=0}^{m}\sum_{k=0}^{m-j}(-1)^{k} \frac{D_{k}^{(s)}r_{j}y_{j}a_{m}}{t_{k+j}}\\
& =\displaystyle\bigg[ \sum_{n=0}^{m-1}\sum_{k=0}^{n} (-1)^{k} \frac{D_{k}^{(s)}}{t_{k}}a_n + \sum_{k=0}^{m} (-1)^{k} \frac{D_{k}^{(s)}}{t_{k}}a_m \bigg]r_0y_0 \\ & ~~~ +
\displaystyle\bigg[ \sum_{n=0}^{m-1}\sum_{k=0}^{n-1} (-1)^{k} \frac{D_{k}^{(s)}}{t_{k+1}}a_n + \sum_{k=0}^{m-1} (-1)^{k} \frac{D_{k}^{(s)}}{t_{k+1}}a_m \bigg]r_1y_1 + \cdots + \frac{D_{0}^{(s)}}{t_{m}}r_m y_m a_m\\
& =  \bigg[\frac{D_{0}^{(s)}}{t_{0}}a_0 + \Big( \frac{D_{0}^{(s)}}{t_{0}}- \frac{D_{1}^{(s)}}{t_{1}} \Big) \sum_{j=1}^{m}a_j + \sum_{j=2}^{m}(-1)^{j} \frac{D_{j}^{(s)}}{t_{j}} \Big( \sum_{k=j}^{m}a_k \Big) \bigg]r_0 y_0 \\ &  ~~ +
\bigg[\frac{D_{0}^{(s)}}{t_{1}}a_1 + \Big( \frac{D_{0}^{(s)}}{t_{1}}- \frac{D_{1}^{(s)}}{t_{2}} \Big) \sum_{j=2}^{m}a_j + \sum_{j=3}^{m}(-1)^{j-1} \frac{D_{j-1}^{(s)}}{t_{j}} \Big( \sum_{k=j}^{m}a_k \Big)\bigg]r_1 y_1 + \cdots + \frac{r_m a_m}{t_m} D_{0}^{(s)}y_{m}\\
& = \sum_{n=0}^{m}r_n \bigg[\frac{a_{n}}{s_0t_n} + \Big(\frac{D_{0}^{(s)}}{t_n}- \frac{D_{1}^{(s)}}{t_{n+1}} \Big)
\sum_{j=n+1}^{m}a_{j}+ \sum_{j=n+2}^{m}(-1)^{j-n} \frac{D_{j-n}^{(s)}}{t_{j}}\Big(\sum_{k=j}^{m}a_{k}\Big) \bigg]y_{n} \\
&= (Ey)_{m},
\end{align*}
where $E=(e_{mn})$ is the matrix defined above. \\Thus $a \in \big[X(r,s,t;
\Delta)\big]^{\gamma}$ if and only if $ax=(a_nx_n)\in bs$ for
$x\in X(r,s,t; \Delta)$ if and only if
$\Big(\displaystyle\sum_{n=0}^{m}a_n x_n \Big)_{m} \in
l_{\infty}$, i.e. $(Ey)_{m} \in l_{\infty}$, for $y\in X$. Hence by Theorem 4.4(b), we have
 $$ \big[X(r,s,t; \Delta)\big]^{\gamma} = \Gamma.$$
\end{proof}

\subsection{$\beta$-dual and Matrix transformations}
We now first discuss about the $\beta$-dual and then
characterize some matrix transformations. Let $T$ be a triangle and
$X_{T}$ be the matrix domain of $T$.
\begin{thm}\label{1}
{(\cite{JAR}, Theorem 2.6)} Let $X $ be a BK space with AK property and
$R=S^{t}$, the transpose of $S$, where $S=(s_{jk})$ is the inverse of the matrix $T$. Then
$a \in (X_{T})^{\beta} $ if and only if
 $a \in (X^{\beta})_{R}$ and $W \in (X, c_{0})$, where the triangle $W$ is defined by $w_{mk}  = \sum\limits_{j=m}^{\infty}a_{j}s_{jk}$.
 Moreover if $a \in (X_{T})^{\beta}$, then
$$ \sum\limits_{k=0}^{\infty}a_{k}z_{k} = \sum\limits_{k=0}^{\infty} R_{k}(a)T_{k}(z)  \qquad \forall ~z \in X_{T}.$$
\end{thm}
\begin{rmrk}{(\cite{JAR}, Remark 2.7)}\label{r_1}
The conclusion of the Theorem $\ref{1}$ is also true for $X=l_\infty$.
\end{rmrk}

\begin{rmrk}(\cite{MAL1}, \cite{JAR}) \label{r_2}
We have $a \in (c_{T})^{\beta}$ if and only if $Ra \in l_{1}$ and $W \in (c,c)$.
Moreover, if $a \in (c_{T})^{\beta}$ then we have for all $z \in c_{T}$
$$\sum\limits_{k=0}^{\infty}a_{k}z_{k} = \sum\limits_{k=0}^{\infty}R_{k}(a)T_{k}(z) - \eta \gamma, $$
where $\eta = \displaystyle \lim_{k\rightarrow \infty} T_{k}(z)$ and $\gamma = \displaystyle \lim_{m \rightarrow \infty} \sum\limits_{k=0}^{m}w_{mk}$.
\end{rmrk}

To find $\beta$-duals of the sequence spaces $X(r, s, t; \Delta)$ for $X\in \{l_{\infty}, c, c_0\}$, we list the following sets:
\begin{align*}
 &B_1 = \Big \{ a \in w : \sum\limits_{k=0}^{\infty} |R_{k}(a)| < \infty \Big \}\\
 &B_2 = \Big \{ a \in w : \lim_{m \rightarrow \infty}w_{mk} =0 ~~{\rm for~ all}~ k \Big \}\\
 &B_3 = \Big \{ a \in w : \sup_{m}\sum\limits_{k=0}^{\infty}|w_{mk}| < \infty \Big \}\\
 &B_4 = \Big \{ a \in w :  \lim_{m \rightarrow \infty} \sum\limits_{k=0}^{m}|w_{mk}|=0 \Big \}\\
 &B_5 = \Big \{ a \in w :  \lim_{m \rightarrow \infty}w_{mk} ~~ {\rm~ exists~ for~ all }~ k \Big \}\\
 &B_6 = \Big \{ a \in w :  \lim_{m \rightarrow \infty}\sum\limits_{k=0}^{m}w_{mk} \mbox{~~exists~}\Big \},
\end{align*}
where $ R_{k}(a)
   =   r_{k} \bigg [ \frac{a_{k}}{s_{0}t_{k}} + \Big( \frac{D_{0}^{(s)}}{t_{k}} - \frac{D_{1}^{(s)}}{t_{k+1}} \Big) \sum\limits_{j= k+1}^{\infty} a_{j} +
\sum\limits_{l=2}^{\infty} (-1)^{l} \frac{D_{l}^{(s)}}{t_{l+k}}
\sum\limits_{j=k+l}^{\infty} a_{j}\bigg ] $, $R(a)= (R_{k}(a))_{k}$ and \\
$w_{mk} = r_{k} \bigg [\sum\limits_{l=0}^{m-k}(-1)^{l}
\frac{D_{l}^{(s)}}{t_{l+k}}\sum\limits_{j= m}^{\infty} a_{j} +
\sum\limits_{l=m-k+1}^{\infty}(-1)^{l} \frac{D_{l}^{(s)}}{t_{l+k}}
\sum\limits_{j= k+l}^{\infty} a_{j} \bigg ]$.

\begin{thm}
We have $[c_{0}(r, s, t; \Delta)]^\beta = B_{1} \bigcap B_{2}\bigcap B_{3}$, $[l_{\infty}(r, s, t; \Delta)]^\beta=B_{1} \bigcap B_{4} $ and $[c(r, s, t; \Delta)]^\beta=B_{1} \bigcap B_{3}\bigcap B_{5} \bigcap B_{6}$.
\end{thm}
\begin{proof}
Here the matrix $T=A(r, s, t).\Delta=(t'_{nk})$, where
\begin{displaymath}
t'_{nk} = \left\{
\begin{array}{ll}
\frac{1}{r_n}[s_{n-k}t_k-s_{n-k+1}t_{k+1}] & \mbox{if}~~   0\leq k<n \\
\frac{s_0t_n}{r_n} &\mbox{if}~~ k=n\\
0 & \mbox{if} ~~ k>n.
\end{array}
\right.
\end{displaymath} So,  $T^{-1}= {(A(r, s, t).\Delta)}^{-1}=\Delta^{-1}.{A(r, s, t)}^{-1}$. Let $S=(s_{jk})$ be the inverse of $T$. Then we easily get
\begin{displaymath}
  s_{jk}  = \left\{
     \begin{array}{ll}
        {  \sum\limits_{l=0}^{j-k}(-1)^{l} \frac{D_{l}^{(s)}}{t_{l+k}} r_{k}} &  \mbox{if}~~ \quad 0\leq k \leq j\\
        0              & \mbox{if}~~ \quad  k > j.
     \end{array}
   \right.
\end{displaymath}

To compute $\beta$-duals, we first determine $W=(w_{mk})$ and $R(a)=(R_{k}(a))$, where $R=S^t$.
\begin{align*}
   R_{k}(a)  & = \sum\limits_{j =k}^{\infty} a_{j}s_{jk}\\
  & = \frac{D_{0}^{(s)}}{t_{k}} r_{k}a_k  +  \sum\limits_{j =k+1}^{\infty} \sum\limits_{l=0}^{j-k}(-1)^{l} \frac{D_{l}^{(s)}}{t_{l+k}} r_{k}a_{j}\\
  & = \frac{D_{0}^{(s)}}{t_{k}} r_{k}a_k  +  \sum\limits_{l=0}^{1}(-1)^{l} \frac{D_{l}^{(s)}}{t_{l+k}} r_{k}a_{k+1} + \sum\limits_{l=0}^{2}(-1)^{l} \frac{D_{l}^{(s)}}{t_{l+k}} r_{k}a_{k+2} + \cdots \\
  & =   r_{k} \bigg [ \frac{a_{k}}{s_{0}t_{k}} + \Big( \frac{D_{0}^{(s)}}{t_{k}} - \frac{D_{1}^{(s)}}{t_{k+1}} \Big) \sum\limits_{j= k+1}^{\infty} a_{j} +
\sum\limits_{l=2}^{\infty} (-1)^{l} \frac{D_{l}^{(s)}}{t_{l+k}}
\sum\limits_{j=k+l}^{\infty} a_{j} \bigg ]
\end{align*}and
\begin{align*}
w_{mk} & = \sum\limits_{j=m}^{\infty}a_{j}s_{jk} \\
& = \sum\limits_{j=m}^{\infty}\sum\limits_{l=0}^{j-k}(-1)^{l} \frac{D_{l}^{(s)}}{t_{l+k}} r_{k}a_j\\
& = r_{k} \bigg [\sum\limits_{l=0}^{m-k}(-1)^{l}
\frac{D_{l}^{(s)}}{t_{l+k}}\sum\limits_{j= m}^{\infty} a_{j} +
\sum\limits_{l=m-k+1}^{\infty}(-1)^{l} \frac{D_{l}^{(s)}}{t_{l+k}}
\sum\limits_{j= k+l}^{\infty} a_{j} \bigg ].
\end{align*}
Using Theorem $\ref{1}$ and Remark \ref{r_1} \& \ref{r_2}, we have
$[c_{0}(r, s, t; \Delta)]^\beta = B_{1} \bigcap B_{2}\bigcap B_{3}$, $[l_{\infty}(r, s, t; \Delta)]^\beta=B_{1} \bigcap B_{4} $ and $[c(r, s, t; \Delta)]^\beta=B_{1} \bigcap B_{3}\bigcap B_{5} \bigcap B_{6}$.
\end{proof}

\begin{thm} {(\cite{JAR}, Theorem 2.13)}{\label{2}}
Let $X $ be a BK space with AK property, $Y$ be an arbitrary subset of $w$ and $R = S^{t}$. Then $A \in (X_{T}, Y)$ if and only if $B^{A} \in (X, Y)$ and $W^{A_{n}} \in (X, c_{0})$ for all $n =0, 1,2, \cdots$, where $B^A$ is the matrix with rows $B_n^A=R(A_n)$, $A_n$ are the rows of $A$ and the triangles $W^{A_n}$ are defined by
\begin{displaymath}
  w^{A_{n}}_{mk}  = \left\{
     \begin{array}{ll}
        {  \sum\limits_{j=m}^{\infty}a_{nj}s_{jk} }& : \quad 0\leq k \leq m \\
        0              & :  \quad  k > m.
     \end{array}
   \right.
\end{displaymath}
\end{thm}
\begin{thm} {(\cite{JAR})}
Let $Y$ be any linear subspace of $w$. Then $A \in (c_{T}, Y)$ if and only if $R_k(A_{n}) \in (c_{0}, Y)$ and $W^{A_{n}} \in (c, c)$ for all $n$ and
$R_k(A_{n})e - (\gamma_{n}) \in Y$, where $\gamma_{n} = \displaystyle \lim_{m \rightarrow \infty} \sum\limits_{k=0}^{m}w^{A_n}_{mk}$ for $n=0,1,2\cdots$ and $e=(1, 1, 1, \cdots)$.\\
Moreover, if $A \in (c_{T}, Y )$ then we have
$$ Az = R_k(A_{n})(T(z)) - \eta (\gamma_{n}) \quad  {\rm for~ all }~ z \in c_T, ~{\rm where}~ \eta = \displaystyle \lim_{k \rightarrow \infty}T_{k}(z).$$
\end{thm}
 To characterize the matrix transformations $A\in (X(r, s,t; \Delta), Y)$ for $X, Y\in \{l_{\infty}, c, c_0\}$, we consider the following conditions:\\
 \begin{align}
 &\displaystyle\sup_{n}\sum\limits_{k=0}^{\infty}|R_{k}(A_{n})| < \infty\\
 &\displaystyle\lim_{n \rightarrow \infty } R_{k}(A_{n}) =0 \quad  \mbox{~for all~} k
 \end{align}
 \begin{align}
&\displaystyle\sup_{m} \sum\limits_{k=0}^{m}|w_{mk}^{A_{n}}| < \infty \quad \mbox{~for all~} n \\
&\displaystyle\lim_{m \rightarrow \infty}w_{mk}^{A_{n}} =0 \mbox{~for all~} n, k \\
&\displaystyle\lim_{n \rightarrow \infty } R_{k}(A_{n})  \mbox{~exists for all~} k\\
&\displaystyle\lim_{n \rightarrow \infty}\sum\limits_{k=0}^{\infty}|R_{k}(A_{n})| = 0\\
&\displaystyle\lim_{m \rightarrow \infty} \sum\limits_{k=0}^{m}|w_{mk}^{A_{n}}| =0  \quad \mbox{~for all~} n \\
&\displaystyle\lim_{n \rightarrow \infty } \sum_{k=0}^{\infty}\Big|R_{k}(A_{n})- \lim_{n \rightarrow \infty }R_{k}(A_{n})\Big| =0\\
&\displaystyle\lim_{m \rightarrow \infty}w_{mk}^{A_{n}} \mbox{~exists for all~} k, ~n \\
&\displaystyle\lim_{m \rightarrow \infty} \sum\limits_{k=0}^{m} w_{mk}^{A_{n}}\mbox{~exists~}\mbox{~for all~} n\\
  &R_{k}(A_{n})e - (\gamma_{n}) \in c_{0}  \quad  \mbox{~for all~}  \gamma_{n}, ~n=0, 1, 2, \cdots\\
 &R_{k}(A_{n})e - (\gamma_{n}) \in l_{\infty} \quad  \mbox{~for all~} \gamma_{n},~ n=0, 1, 2, \cdots\\
 &R_{k}(A_{n})e - (\gamma_{n}) \in c  \quad  \mbox{~for all~}  \gamma_{n},~ n=0, 1, 2, \cdots,
 \end{align}
 where $\gamma_{n}= \displaystyle\lim_{m \rightarrow \infty} \sum\limits_{k=0}^{m} w_{mk}^{A_{n}}$,\\
$R_{k}(A_{n}) = r_{k} \Big [ \frac{a_{nk}}{s_{0}t_{k}} + \Big( \frac{D_{0}^{(s)}}{t_{k}} - \frac{D_{1}^{(s)}}{t_{k+1}} \Big) \sum\limits_{j= k+1}^{\infty} a_{nj} +
\sum\limits_{l=2}^{\infty} (-1)^{l} \frac{D_{l}^{(s)}}{t_{l+k}} \sum\limits_{j=k+l}^{\infty} a_{nj} \Big ]$
and\\
$w^{A_{n}}_{mk} = r_{k} \bigg [\sum\limits_{l=0}^{m-k}(-1)^{l}
\frac{D_{l}^{(s)}}{t_{l+k}}\sum\limits_{j= m}^{\infty} a_{nj} +
\sum\limits_{l=m-k+1}^{\infty}(-1)^{l} \frac{D_{l}^{(s)}}{t_{l+k}}
\sum\limits_{j= k+l}^{\infty} a_{nj} \bigg ].$

\begin{thm}
(a)  $A \in (c_{0}(r, s,t; \Delta), c_{0})$ if and only if the conditions $(4.12), (4.13),(4.14)$ and $(4.15)$ hold.\\
(b) $A \in (c_{0}(r, s,t; \Delta), c)$ if and only if the conditions $(4.12), (4.14), (4.15)$ and $(4.16)$ hold.\\
(c) $A \in (c_{0}(r, s,t; \Delta), l_{\infty})$ if and only if the conditions $(4.12), (4.14)$ and $(4.15)$ hold.
\end{thm}
\begin{proof}We prove only the part (a) of this theorem. The other parts follow in a similar way. For this, we first compute the matrices $B^A= R_{k}(A_{n})$ and $W^{A_{n}} =(w_{mk}^{A_{n}})$ for $n =0, 1, 2, \cdots$ of Theorem \ref{2} to determine the conditions $B^A\in (c_0, c_0)$ and $W^{A_{n}}\in (c_0, c_0)$.
Using the same lines of proof as used in Theorem 4.8, we have
\begin{align*}
R_{k}(A_{n}) &= \sum\limits_{j=k}^{\infty}s_{jk}a_{nj}\\
& = \sum\limits_{j=k+1}^{\infty} \sum\limits_{l=0}^{j-k}(-1)^{l} \frac{D_{l}^{(s)}}{t_{l+k}} r_{k}a_{nj} + \frac{D_{0}^{(s)}}{t_{k}} r_{k}a_{nk}\\
& = r_{k} \Big [ \frac{a_{nk}}{s_{0}t_{k}} + \Big( \frac{D_{0}^{(s)}}{t_{k}} - \frac{D_{1}^{(s)}}{t_{k+1}} \Big) \sum\limits_{j= k+1}^{\infty} a_{nj} +
\sum\limits_{l=2}^{\infty} (-1)^{l} \frac{D_{l}^{(s)}}{t_{l+k}} \sum\limits_{j=k+l}^{\infty} a_{nj} \Big ]
\end{align*}
and
\begin{align*}
w^{A_{n}}_{mk} & = \sum\limits_{j=m}^{\infty}a_{nj}s_{jk} \\
& = r_{k} \bigg [\sum\limits_{l=0}^{m-k}(-1)^{l}
\frac{D_{l}^{(s)}}{t_{l+k}}\sum\limits_{j= m}^{\infty} a_{nj} +
\sum\limits_{l=m-k+1}^{\infty}(-1)^{l} \frac{D_{l}^{(s)}}{t_{l+k}}
\sum\limits_{j= k+l}^{\infty} a_{nj} \bigg ].
\end{align*}
Using Theorem \ref{2}, we have $A\in (c_{0}(r, s,t; \Delta), c_{0})$ if and only if the conditions $(4.12), (4.13), (4.14)$ and $(4.15)$ hold.
\end{proof}

\par We can also obtain the following results.
\begin{cor}
(a)  $A \in (l_{\infty}(r, s,t; \Delta), c_{0})$ if and only if the conditions $(4.17)$ and $(4.18)$ hold.\\
(b) $A \in (l_{\infty}(r, s,t; \Delta), c)$ if and only if the conditions $(4.12), (4.16),(4.18)$ and $(4.19)$ hold.\\
(c) $A \in (l_{\infty}(r, s,t; \Delta), l_{\infty})$ if and only if the conditions $(4.12)$ and $(4.18)$ hold.
\end{cor}
\begin{cor}
(a)  $A \in (c(r, s,t; \Delta), c_{0})$ if and only if the conditions $(4.12), (4.13), (4.14), (4.20), (4.21)$ and $(4.22)$ hold.\\
(b) $A \in (c(r, s,t; \Delta), c)$ if and only if the conditions  $(4.12), (4.14), (4.16), (4.20), (4.21)$ and $(4.24)$ hold.\\
(c) $A \in (c(r, s,t; \Delta), l_{\infty})$ if and only if the conditions  $(4.12), (4.14), (4.20), (4.21)$ and $(4.23)$ hold.
\end{cor}
\section{Some applications}
In this section, we justify our results in some special cases. Also, we illustrate the results related to matrix transformations given by Djolovic\cite{DJO}, Polat and Altay \cite{POL1}, Aydin and Ba\c{s}ar \cite{AYD} etc. \\
 $(I)$~ In particular, if we choose $r_n = \frac{1}{n!}$, $t_n = \frac{\alpha^{n}}{n!}$ and $s_{n}=\frac{(1- \alpha)^{n}}{n!}$,
 where $0< \alpha < 1$ then the sequence spaces $ X(r, s,t; \Delta)$ for $X\in\{ l_\infty, c_0, c\}$ reduce to the Euler difference sequence spaces $e^\alpha_\infty(\Delta)$,
 $e^\alpha_0(\Delta)$ and $e^\alpha_{c}(\Delta)$ respectively \cite{POL1}. By the above choice of $r, s$ and $ t$, we
 have $D_{0}^{(s)} =1$, $D_{1}^{(s)} =(1- \alpha)$, $D_{2}^{(s)} =\frac{(1-\alpha)^2}{2!}$, $D_{3}^{(s)} =\frac{(1-\alpha)^3}{3!}$ and so on. Therefore $D_{k}^{(s)} =\frac{(1-\alpha)^k}{k!}$, $k \in \mathbb{N}_0$.
 Thus the $\alpha$-dual of the Euler difference sequence spaces is the set
 \begin{align*}
 &\displaystyle\Big \{ a \in w : \sup_{K \in \mathcal{F}}\sum_{n}\bigg| \sum_{j \in K} \sum_{k =j}^{n} (-1)^{k-j}\frac{(1 -\alpha)^{k-j}} {(k-j)!}.
 \frac{1}{j!} \alpha^{-k} k! a_{n}\bigg | < \infty \Big \}\\
 &~= \displaystyle\Big \{ a \in w : \sup_{K \in \mathcal{F}}\sum_{n}\bigg| \sum_{j \in K} \sum_{k =j}^{n} (-1)^{k-j}\binom{k}{j}(1 -\alpha)^{k-j} \alpha^{-k}a_{n}
\bigg| < \infty \Big \}.
 \end{align*}
Here we illustrate that how the characterization of matrix transformation $A \in (e_{\infty}^{\alpha}(\Delta), l_{\infty})$ can be obtained with the help of Corollary 4.1(c)\\
\begin{align*}
 R_{k}(A_{n}) &= \sum_{j=k}^{\infty} a_{nj}s_{jk}\\
 & = r_k \bigg [ \frac{a_{nk}}{s_{0}t_k} + \Big( \frac{D_{0}^{(s)}}{t_k} - \frac{D_{1}^{(s)}}{t_{k+1}}\Big)\sum_{j=k+1}^{\infty} a_{nj} +
  \sum_{l=2}^{\infty} (-1)^{l} \frac{D_{l}^{(s)}}{t_{k+l}} \sum_{j= k+l}^{\infty} a_{nj}\bigg]\\
 & =  \frac{a_{nk}}{\alpha_k} + \Big( \frac{1}{\alpha^k} - \frac{1}{\alpha^{k+1}}\Big)\sum_{j=k+1}^{\infty} a_{nj} + \sum_{l=2}^{\infty} (-1)^{l} \binom{l+k}{l}\frac{(1 - \alpha)^{l}}{\alpha^{k+l}} \sum_{j= k+l}^{\infty} a_{nj}.
\end{align*}
\begin{align*}
 w_{mk}^{A_{n}} &= \sum_{j=m}^{\infty} a_{nj}s_{jk}\\
 & = r_k \bigg [\sum_{l=0}^{m-k} (-1)^{l} \frac{D_{l}^{(s)}}{t_{k+l}} \sum_{j=m}^{\infty} a_{nj}
  \sum_{l=m-k+1}^{\infty} (-1)^{l} \frac{D_{l}^{(s)}}{t_{k+l}} \sum_{j= k+l}^{\infty} a_{nj}\bigg ]\\
& = \frac{1}{k!}\bigg [ \sum_{l=0}^{m-k} (-1)^{l} \frac{(1 -
\alpha)^l (l+k)!}{l! \alpha^{l+k}}\sum_{j=m}^{\infty} a_{nj}
+\sum_{m-k+1}^{\infty} (-1)^{l}
\frac{(1 - \alpha)^l (l+k)!}{l! \alpha^{l+k}}\sum_{j=k+l}^{\infty} a_{nj}  \bigg ]\\
& = \bigg[ \sum_{l=0}^{m-k} (-1)^{l} \binom{l+k}{l}\frac{(1 -
\alpha)^l }{\alpha^{l+k}}\sum_{j=m}^{\infty} a_{nj}
+\sum_{m-k+1}^{\infty} (-1)^{l} \binom{l+k}{l}) \frac{(1 - \alpha)^l
}{\alpha^{l+k}}\sum_{j=k+l}^{\infty} a_{nj} \bigg ].
\end{align*}
 So $A \in (e_{\infty}^{\alpha}(\Delta), l_{\infty})$ if and only if\\
 $$\displaystyle\sup_{n} \sum_{k=0}^{\infty} |R_{k}(A_n)| < \infty$$
 and $$ \lim_{m \rightarrow \infty} \sum_{k=0}^{m}|w_{mk}^{A_{n}}| =0 \quad {\rm for~ all~} n.$$

 $(II)$~ We choose $s_n = 1$ $ \forall ~n$, $r_n = (n+1)$, $t_n = 1+ \alpha^{n}$, where $0< \alpha < 1$ then the sequence spaces $X(r, s, t; \Delta)$ for $X \in \{ l_{\infty}, c, c_{0} \}$
 reduce to $a_{\infty}^{\alpha}(\Delta)$, $a_{c}^{\alpha}(\Delta)$ and $a_{0}^{\alpha}(\Delta)$ respectively. With this choice $D_0^{(s)}=1=D_1^{(s)}$ and
 $D_k^{(s)}=0$ for all $k\geq2$. Therefore the matrices $R_{k}(A_{n})$ and $W^{A_n} =(w_{mk}^{A_n})$ become
 $$ R_{k}(A_{n}) = (k+1) \bigg [\frac{a_{nk}}{1 + \alpha^{k}} + \Big( \frac{1}{1 + \alpha^{k}} - \frac{1}{1 + \alpha^{k +1}}  \Big )
 \sum_{j =k+1}^{\infty}a_{nj}\bigg ],$$
 and
 \begin{align*}
 w_{mk}^{A_{n}} &= r_{k} \bigg[ \Big( \frac{D_{0}^{(s)}}{ t_k} - \frac{D_{1}^{(s)}}{ t_{k+1}} \Big) \sum_{j=m }^{\infty}a_{nj} -
 \frac{D_{1}^{(s)}}{ t_{m+1}} \sum_{j=m+1 }^{\infty}a_{nj}\bigg]\\
&=  (k+1) \bigg [ \Big( \frac{1}{ 1 +\alpha^{k}} - \frac{1}{ 1
+\alpha^{k+1}} \Big) \sum_{j=m }^{\infty}a_{nj} -  \frac{1}{ 1
+\alpha^{m+1}} \sum_{j=m+1 }^{\infty}a_{nj} \bigg ].
\end{align*} By evaluating, we have
$\displaystyle \sum_{k=0}^{m}|w_{mk}^{A_{n}}| = \sum_{k=0}^{m} (k+1)  \Big( \frac{1}{ 1 +\alpha^{k}} - \frac{1}{ 1 +\alpha^{k+1}}
\Big) \Big|\sum_{j=m }^{\infty}a_{nj}\Big| + |a_{nm}| \frac{m+1}{1+ \alpha^{m+1}}$.\\

Therefore by Corollary 4.1(a), we have $A \in (a_{\infty}^{\alpha}(\Delta), c_{0})$ if and only if
$$ \displaystyle \lim _{n \rightarrow \infty} \sum_{k=0}^{\infty}|R_{k}(A_{n})| =0$$ and
$$ \displaystyle \lim _{m \rightarrow \infty}\sum_{k=0}^{m}|w_{mk}^{A_{n}}| =0 \mbox{~for all~} n.$$

\end{document}